\documentclass[12pt, reqno]{amsart}
\usepackage{amsmath,amssymb,amsthm}

 2

\usepackage[colorlinks=true,linkcolor=blue,citecolor=blue,pdfpagelabels=false]{hyperref}

\newtheorem{thm}{Theorem}[section]
\newtheorem{lem}[thm]{Lemma}
\newtheorem{prop}[thm]{Proposition}
\newtheorem{cor}[thm]{Corollary}

\newcommand{\thmref}[1]{Theorem~\ref{#1}}
\newcommand{\lemref}[1]{Lemma~\ref{#1}}

\newcommand{\propref}[1]{Proposition~\ref{#1}}
\newcommand{\corref}[1]{Corollary~\ref{#1}}

\theoremstyle{remark}
\newtheorem{rmk}{Remark}[section]

\newenvironment{acknowledgements}{\bigskip\textbf{Acknowledgements.}}{}

\renewcommand{\geq}{\geqslant}
\renewcommand{\leq}{\leqslant}

\begin{document}
\title[Simultaneous sign changes]
{The first simultaneous sign change for Fourier coefficients of Hecke-Maass forms}
\author[M. Kumari, J. Sengupta]{Moni Kumari and Jyoti Sengupta}
\address{ 
School of Mathematics, Tata Institute of Fundamental Research, 1, Homi Bhabha Road, Mumbai, 400005, India.}
\email{moni@math.tifr.res.in}

\address{Jyoti Sengupta, Department of Mathematics, Vivekananda University, Belur, West Bengal, India.}
\email{sengupta@math.tifr.res.in}

\date{\today}


\maketitle

\begin{abstract}
Let $f$ and $g$ be two Hecke-Maass cusp forms of weight zero for $SL_2(\mathbb Z)$ with Laplacian eigenvalues $\frac{1}{4}+u^2$ and $\frac{1}{4}+v^2$, respectively. Then both have real Fourier coefficients say, $\lambda_f(n)$ and $\lambda_g(n)$, and we may normalize $f$ and $g$ so that $\lambda_f(1)=1=\lambda_g(1)$. In this article, we first prove that the sequence $\{\lambda_f(n)\lambda_g(n)\}_{n \in \mathbb{N}}$ has infinitely many sign changes. Then we derive a bound for the first negative coefficient for the same sequence in terms of the Laplacian eigenvalues of $f$ and $g$. 
\end{abstract}

\section{Introduction and main results}
In recent times studying questions related to sign changes for Fourier coefficients of an automorphic form, in particular modular form and Maass form are very much in trend.
For $k\in \mathbb Z $, we denote by $S_k$ the space of cusp forms (holomorphic) of weight $k$ for the full modular group$SL_2(\mathbb Z)$.
Let $f \in S_k$ be a Hecke cusp form with Fourier coefficients $\lambda_f(n)$, then
using a classical theorem of Landau and certain analytic properties of the associated $L$-function $L_{f}(s)$ of $f$, one can show that the sequence $\{\lambda_f(n)\}_{n\in \mathbb{N}}$ has infinitely many sign changes. 
In 1969, Siegel \cite{Siegel} proved that if $k\equiv 2\pmod{4},$ then if one goes upto dim$S_{k}+1$ then there will be at least one negative Fourier coefficient $\lambda_{f}(n)$. In other words, let $n_f$ be the smallest positive integer $n$ such that $\lambda_f(n)<0,$ then Siegel's result says that $n_f\leq \mbox{dim}S_{k}+1.$ The best known bound for $n_f$ so far (in our knowledge) is due to Matom$\ddot{a}$ki \cite{Mato} and she proved that $$n_f\ll k^{3/4}.$$

The analogous question about infinitely many sign changes for Fourier coefficients of a Hecke-Maass form $f$ was pointed out by Knopp, Kohnen and Pribitkin \cite{kkp} at the end of their paper. In 2010, Yan \cite{yan} proved that, in this case, there exist a positive integer
$$n\ll (3+|u|^2)^{1/2-\delta}$$ 
such that $\lambda_f(n)<0$, where $\delta$ is a positive constant. 
Later, these problems have been generalised by several authors for other class of automorphic forms.

Instead of one cusp form, one can take more than one cusp form (holomorphic or non-holomorphic) and ask the analogous questions. 
 Regarding the case of two holomorphic 
cusp forms, there are a number of interesting results (provided as examples) in \cite{pri}.
The authors of \cite{gun} and \cite{koh3} worked with two (holomorphic) cusp forms and essentially they proved that the sequence $\{\lambda_f(n)\lambda_{g}(n)\}_{n\in \mathbb{N}}$ 
has infinitely many sign changes provided $\lambda_{f}(1)\lambda_{g}(1)\neq 0,$ where $\lambda_f(n)$ and $\lambda_{g}(n)$ are the normalized Fourier coefficients of cusp forms $f$ and $g$, respectively.
Recently, the first author together with M. Ram Murty \cite{Moni} proved that the sequence $\{\lambda_f(n)\lambda_{g}(n)\}_{n\in \mathbb{N}}$ has infinitely many sign changes without assuming the hypothesis $\lambda_{f}(1)\lambda_{g}(1)\neq 0$ and there they also gave a quantitative result for simultaneous sign change in short intervals. 
In \cite{llw} Lau, Liu and Wu investigated the problem of the first sign change for
the sequence $\{\lambda_f(n)\lambda_{g}(n)\}_{n\in \mathbb{N}}$ when $f$ and $g$ are holomorphic cusp forms.

In this article, we study the analogous questions for two distinct Hecke-Maass cusp forms.
We start by showing that the sequence $\{\lambda_f(n)\lambda_{g}(n)\}_{n\in \mathbb{N}}$ has infinitely many sign changes.

\begin{thm}\label{signc}
Let $f$ and $g$ be two normalized Hecke-Maass cusp forms with Fourier coefficients $\lambda_f(n)$ and $\lambda_g(n)$ respectively. Then for any sufficiently large $x$, the interval $(x,2x]$ contain positive integers
\begin{enumerate}
\item 
$n$ such that $\lambda_f(n)\lambda_g(n)<0,$
and 
\item
$m$ such that $\lambda_f(m)\lambda_g(m)>0.$
\end{enumerate}
\end{thm}
Therefore, we have an immediate corollary.
\begin{cor}
There are infinitely many sign changes in the sequence $\{\lambda_f(n)\lambda_g(n)\}_{n\in \mathbb{N}}$. 
\end{cor}
Now we state our next result concerning the problem for the first negative coefficient in the sequence
$\{\lambda_f(n)\lambda_{g}(n)\}_{n\in \mathbb{N}}$. Let $n_{f,g}$ denotes the smallest positive integer $n$ such that $\lambda_f(n)\lambda_{g}(n)<0$. We bound the size of $n_{f,g}$ in terms of the Laplacian eigenvalues of $f$ and $g$ by proving the following theorem (see also preliminaries).

\begin{thm}\label{firstn}
Let $f$ and $g$ be two normalized Hecke-Maass cusp forms (of weight zero) for $SL_2(\mathbb{Z})$ with Laplacian eigenvalues $\frac{1}{4}+u^2$ and $\frac{1}{4}+v^2$, respectively. Then 
\begin{equation*}
n_{f,g}\ll \max \left \{ {{\rm exp}} \big(\tilde{c}\log ^2(\sqrt{{\rm{max}}\{q(sym^2 f),q(sym^2 g)\}})\big),(1+u+v)^{5.34}\right\},
\end{equation*}
where the implied constant is absolute. Here $\tilde{c}$ is an absolute constant and $q(sym^2 f)$, $q(sym^2 g)$ are the analytic conductor of the symmetric square $L$-function associated to $f$ and $g$ respectively. We have $q(sym^2 f)=3(3+u)^2$ and $q(sym^2 g)=3(3+v)^2$.
\end{thm}
\begin{rmk}
If we assume the Ramanujan's conjecture for Fourier coefficients of $f$ and $g$, i.e., 
$\lambda_f(n), \lambda_g(n)\ll_\epsilon n^\epsilon$, then following the method of this paper one can show that 
$$n_{f,g}\ll_{\epsilon}  \max \left \{{\rm exp} \big(\tilde{c}\log ^2(\sqrt{{\rm{max}}\{q(sym^2 f),q(sym^2 g)\}})\big),(1+u+v)^{1+\epsilon}\right\}, ~~\forall~ \epsilon>0.$$
\end{rmk}
Throughout the paper, we will work with Hecke-Maass forms of weight zero for the group $SL_2(\mathbb Z)$.
Also $p$ is always assumed to be a prime number.
\section{Preliminaries}
In this section we collect the relevant facts that will be needed in various stages of our proofs.

Let $f:\mathbb H\rightarrow \mathbb C$ be a normalized (i.e., $\lambda_f(1)=1$) Hecke-Maass cusp form (of weight zero) for $SL_2(\mathbb Z)$ with Laplacian eigenvalue $\frac{1}{4}+u^2$, and having Fourier expansion
\begin{equation*}
f(x+iy)=\sqrt{y}\sum_{n\neq 0}\lambda_f(n)K_{iu}(2\pi|n|y)e(nx).
\end{equation*}
Here $e(x)=e^{2\pi ix}$ and $K_{iu}$ is the MacDonald-Bessel function. Notice that the spectral parameter $u$ in this case is known to be a non-zero real number 
(i.e., $\frac{1}{4}+u^2> \frac{1}{4}$). Therefore, we can assume that $u>0$.
Like holomorphic Hecke eigenforms, Fourier coefficients of $f$ are also multiplicative in the following sense.
\begin{equation}\label{hr}
\lambda_f(n)\lambda_f(m)=\sum_{d|(m,n)}\lambda_f\left(\frac{mn}{d^2}\right).
\end{equation}
In particular, $\lambda_f(p)^2=\lambda_f(p^2)+1$ for a prime $p$. The Kim-Sarnak \cite{ks} bound states that 
\begin{equation}\label{ks}
|\lambda_f(n)|\leq d(n)n^{\frac{7}{64}},
\end{equation}
 where
$d(n)$ is the divisor function. As in the holomorphic set-up, one expects $|\lambda_f(n)|
\leq d(n)$, which is known as Ramanujan's estimate.

Let $\iota:\mathbb H\rightarrow \mathbb H$ be the antiholomorphic involution
$$\iota(x+iy)=-x+iy.$$ Then $f$ and $f\circ\iota$ have the same Laplacian eigenvalue. Since $\iota^2=1$ therefore its eigenvalues are $\pm 1$.
If $f\circ \iota=f$, we call $f$ even. In this case, $\lambda_f(n)=\lambda_f(-n)$. If $f\circ \iota=-f$, we call $f$ odd and $\lambda_f(n)=-\lambda_f(-n)$. Moreover, we say $f$ has parity $0$ or $1$ according to whether $f$ is even or odd, respectively.

Let $f$ and $g$ be two normalized Hecke-Maass eigenforms for the group $SL_2(\mathbb{Z})$ with Laplacian eigenvalues $\frac{1}{4}+u^2$ and $\frac{1}{4}+v^2$ respectively. Let $\lambda_f(n)$ and $\lambda_g(n)$ denote the $n$-th Fourier coefficient of $f$ and $g$, respectively.
Furthermore, assume $\delta, \eta \in \{0,1\}$ be the parity of $f$ and $g$, respectively.

The Rankin-Selberg $L$-function associated to $f$ and $g$ is defined by an absolutely convergent Euler product
\begin{equation}\label{rsp}
L(f\otimes g,s)=\prod_{p}\prod_{i,j=1,2}\big(1-\alpha_{f,i}(p)\alpha_{g,j}(p)p^{-s}\big)^{-1} ~~~~~{\rm for }~~{\rm{Re}}(s)\gg 1,
\end{equation}
where $\alpha_{f,i}(p)$ and $\alpha_{g,i}(p)$, $i=1,2$ are local parameters 
of $f$ and $g$, respectively. In this half-plane it is also given by an absolutely convergent Dirichlet series
\begin{equation}\label{rss}
L(f\otimes g,s)=\zeta(2s)\sum_{n=1}^{\infty}\frac{\lambda_f(n)\lambda_g(n)}{n^s}.
\end{equation}
For simplicity, we write 
\begin{equation}\label{rf}
R(f,g,s):=\sum_{n=1}^{\infty}\frac{\lambda_f(n)\lambda_g(n)}{n^s}.
\end{equation}

\begin{thm}\cite[Theorem 7.3]{liu}\label{rst}
Let $f$ and $g$ be as above. Then
\begin{enumerate}
\item
the product \eqref{rsp} and the series \eqref{rss} are absolutely convergent for ${\rm{Re}}(s)>1$,
\item
the function $L(f\otimes g,s)$ has analytic continuation and a functional equation. When $f=g$, 
then $L(f\otimes f,s)$ is holomorphic in the entire complex plane except for a simple pole at $s=1$.
\end{enumerate}
\end{thm}
 
The functional equation is actually of the form 
\begin{equation}\label{fe}
\Lambda(f\otimes g,s)=\Lambda(f\otimes g,1-s),
\end{equation}
where $\Lambda(f\otimes g,s)=\gamma(f\otimes g,s)L(f\otimes g,s)$
with
\begin{align*}
\gamma(f\otimes g,s)\!\!
&=\pi^{-2s}\Gamma\bigg(\frac{s+i(u+v)+r}{2}\bigg)\Gamma\bigg(\frac{s+i(u-v)+r}{2}\bigg)\Gamma\bigg(\frac{s-i(u+v)+r}{2}\bigg)\\
& \times \Gamma\bigg(\frac{s-i(u-v)+r}{2}\bigg),
\end{align*} 
where $r=0 ~{\rm or}~1$ according to whether $\delta=\eta$ or not (see, \cite[p. 133]{ik}).

Now, we define the symmetric square $L$-function of $f$ as follows:
\begin{equation}
L(sym^2f,s):=\prod_{p}\prod_{i=0}^{2}\Bigg(1-\alpha_{f,1}^{2-i}(p)\alpha_{f,2}^{i}(p)p^{-s}\Bigg)^{-1}=L(f\otimes f,s)\zeta(s)^{-1}.
\end{equation}
The existence of this $L$-function was established by Gelbart and Jacquet in \cite{gj}. The function $L(sym^2f,s)$ is known
to be entire and $L(sym^2f,1)\neq 0$ (see, \cite[p. 162]{hl}).

The Rankin-Selberg $L$-function $L(sym^2f\otimes sym^2g,s)$ attached to sym$^2f$ and sym$^2g$ is defined as
\begin{equation}\label{rss1}
L(sym^2f\otimes sym^2g,s)=\prod_{p}\prod_{i=0}^{2}\prod_{j=0}^{2}\Bigg(1-\alpha_{f,1}^{2-i}(p)\alpha_{f,2}^{i}(p)\alpha_{g,1}^{2-j}(p)\alpha_{g,2}^{j}(p)p^{-s}\Bigg)^{-1}.
\end{equation}
This function is also known to be entire if $f\neq g$ unless it has a simple pole at $s=1$ and $L(sym^2f\otimes sym^2g,1)\neq 0$ (see, \cite[p. 172]{hl}).
\vspace{.2in}

We now recall a fundamental results from complex analysis which we need further for our proof.

\begin{thm}[Rademacher \cite{rad}]\label{rthm}
Let $h(s)$ be a continuous function on the closed strip $a\leq \sigma \leq b$, holomorphic  and of finite order on $a<\sigma< b$. Further suppose that 
\begin{equation*}
|h(a+it)|\leq E|P+a+it|^\alpha, ~~~~~~~~|h(b+it)|\leq F|P+b+it|^\beta
\end{equation*}
where $E,F$ are positive constants and $P, \alpha,\beta$ are real constants that satisfy
\begin{equation*}
P+a >0, ~~~~~~~\alpha\geq \beta.
\end{equation*}
Then for all $a<\sigma<b$ and for all $t \in \mathbb{R}$, we have
\begin{equation*}
|h(\sigma+it)|\leq \left(E|P+\sigma+it|^\alpha\right)^{\frac{b-\sigma}{b-a}} \left(F|P+\sigma+it|^{\beta}\right)^{\frac{\sigma-a}{b-a}}.
\end{equation*}
\end{thm}

\section{Auxiliary Results }
\begin{lem}\label{lpn}
Let $f$ be a normalized Hecke-Maass cusp form of weight zero for $SL_2(\mathbb Z)$ with Fourier coefficients $\lambda_f(n)$. Then 
\begin{equation}\label{af}
\sum_{p\leq x}\lambda_f^2(p)\log{p}=x +O\left(x\sqrt{q(sym^2 f)} e^{-\frac{c}{162}\sqrt{\log x}}\right),
\end{equation}
where $c$ is an absolute constant.
\end{lem}

\begin{proof}
This is essentially Theorem 5.13 of \cite{ik} for the $L$-function $L(sym^2 f,s)$ combined with the standard Hecke relation
\eqref{hr}. However, for the sake of completeness, here we give an outline of the proof.

Consider the function 
\begin{equation*}
\psi(sym^2 f,x):= \sum_{n\leq x}\Lambda _{sym^2 f}(n),
\end{equation*}
where $\Lambda _{sym^2 f}(n)$ is the $n$-th coefficient of the negative of the logarithmic derivative of the $L$-function $L(sym^2 f,s)$, which are supported only on prime powers.

Now since the Rankin-Selberg $L$-function $L(sym^2 f \otimes sym^2 f,s)$ exists and has a simple pole at $s=1$ (see, \cite[p. 180]{hl}), then 
from \cite[p. 110, exercise 6]{ik}, we have
\begin{equation*}
\psi(sym^2 f,x)=\sum_{p\leq x}\lambda_{sym^2 f}(p)\log p+O\big(\sqrt{x}d^2\log^2 (xq(sym^2 f))\big)
\end{equation*}
where the implied constant is absolute and $d$ is the degree of the $L$-function  $L(sym^2 f,s)$ (notice that $d=3$).
 
Using the fact that $\lambda_{sym^2 f}(p)=\lambda_f(p^2)=\lambda_{f}^2(p)-1$ (by \eqref{hr}), we get
\begin{equation}\label{pnt}
\sum_{p\leq x}\lambda^2_{f}(p)\log (p)=\psi(sym^2 f,x)+\sum_{p\leq x}\log p+O\big(\sqrt{x}d^2\log^2 (xq(sym^2 f))\big).
\end{equation}
If we apply the prime number theorem for the $L$-functions $\zeta(s)$ and $L(sym^2 f,s)$ (see, \cite[equ. 5.52]{ik}, we have
\begin{equation}\label{pnr}
\sum_{p\leq x}\log p=x+O(xe^{-\frac{c}{2}\sqrt{\log x}}),
\end{equation} 
and 
\begin{equation}\label{pns}
\psi(sym^2 f,x)=O\left(x\sqrt{q(sym^2 f)} e^{-\frac{c}{162}\sqrt{\log x}}\right),
\end{equation} 
where $c$ is an absolute constant appearing in the proof of Theorem $5.10$ of \cite{ik}.
After substituting these approximations in \eqref{pnt}, we get
\noindent
\begin{equation*}
\sum_{p\leq x}\lambda^2_{f}(p)\log (p)=x+O\big(\sqrt{x}d^2\log^2 (xq(sym^2 f))\big)+O(xe^{-\frac{c}{2}\sqrt{\log x}})+
O\left(x\sqrt{q(sym^2 f)} e^{-\frac{c}{162}\sqrt{\log x}}\right).
\end{equation*}
Notice that the first two error terms are dominated by the third error term and hence we get the required result.
\end{proof}

\vspace{.2in}
\begin{rmk}
\begin{enumerate}
\item
The approximation formula \eqref{af} has meaning when the error term is smaller than the main term, and this is the case for 
\begin{equation}
x\geq \exp \left(\tilde{c}\log ^2(\sqrt{q(sym^2 f)})\right),
\end{equation}
where $\tilde{c}=\left(\frac{81}{c}\right)^2$.
\item
To write equation \eqref{pns} we have used the facts that the $L$-function $L(sym^2 f,s)$ is entire 
and has no Siegel zero (see, \cite{hl}).
\end{enumerate}
\end{rmk}

\begin{prop}\label{lowerb}
Let $f$ and $g$ be two distinct Hecke-Maass cusp forms for $SL_2(\mathbb Z)$. Assume that $\lambda_f(n)\lambda_g(n)\geq 0$ for all $n\leq x$. Then for 
$x\geq \exp \big(\tilde{c}\log ^2(\sqrt{{\rm{max}}\{q(sym^2 f),q(sym^2 g)\}})\big)$, we have
\begin{equation*}
\sum_{n\leq x}\lambda_f(n)\lambda_g(n)\gg_{f,g} \frac{x^{25/32}}{\log^2x}.
\end{equation*}
\end{prop}

\begin{proof}
From the given hypothesis, we have 
\begin{align*}
\sum_{n\leq x}\lambda_f(n)\lambda_g(n)&\geq \sum_{p,q\leq \sqrt{x}, p\neq q}\lambda_f(pq)\lambda_g(pq)\\
&
=\left(\sum_{p\leq \sqrt{x}}\lambda_f(p)\lambda_g(p)\right)^2 -\sum_{p\leq \sqrt{x}}\lambda_f^2(p)\lambda_g^2(p).
\end{align*}
Further, using the Kim-Sarnak's bound, i.e., $|\lambda_f(p)|,|\lambda_g(p)|\leq2p^{\frac{7}{64}}$ and the fact that $p\leq \sqrt{x}$, we get
\begin{align*}
\left(\sum_{p\leq \sqrt{x}}\lambda_f(p)\lambda_g(p)\right)^2&\geq \frac{1}{16x^{7/32}}
\left(\sum_{p\leq \sqrt{x}}\lambda_f^2(p)\lambda_g^2(p)\right)^2.
\end{align*}

\noindent
Therefore, we have
\begin{align}\label{mb}
\sum_{n\leq x}\lambda_f(n)\lambda_g(n)
\geq \frac{1}{16x^{7/32}}
\left(\sum_{p\leq \sqrt{x}}\lambda_f^2(p)\lambda_g^2(p)\right)^2 -\sum_{p\leq \sqrt{x}}\lambda_f^2(p)\lambda_g^2(p).
\end{align}
Hence to complete the proof, we need an appropriate upper and lower bound for the function $\sum_{p\leq \sqrt{x}}\lambda_f^2(p)\lambda_g^2(p)$. For our purpose, the following trivial upper bound is sufficient which we get using \eqref{ks}.
\begin{equation}\label{mlb}
\sum_{p\leq \sqrt{x}}\lambda_f^2(p)\lambda_g^2(p)\ll x^{\frac{23}{32}}.
\end{equation}

Next we are going to get a lower for the sum $\sum_{p\leq \sqrt{x}}\lambda_f^2(p)\lambda_g^2(p)$.
The Hecke relation \eqref{hr}, gives
\begin{equation*}
\lambda_f(p^2)\lambda_g(p^2)=\lambda_f^2(p)\lambda_g^2(p)-\lambda_f^2(p)-\lambda_g^2(p)+1.
\end{equation*} 
Let $p\leq \sqrt{x}$ be a prime. Then by our hypothesis $\lambda_f(p^2)\lambda_g(p^2)\geq 0$,
i.e., 
\begin{equation*}
\lambda_f^2(p)\lambda_g^2(p)\geq\lambda_f^2(p)+\lambda_g^2(p)-1.
\end{equation*}
Therefore, we have
\begin{align*}
\sum_{p\leq \sqrt{x}}\lambda_f^2(p)\lambda_g^2(p)& \geq\sum_{p\leq \sqrt{x}}\lambda_f^2(p)+\sum_{p\leq \sqrt{x}}\lambda_g^2(p)-\sum_{p\leq \sqrt{x}}1
\end{align*}
Now apply \lemref{lpn} for $f$ and $g$ and the prime number theorem, one can get
\begin{align*}
\sum_{p\leq \sqrt{x}}\lambda_f^2(p)+\sum_{p\leq \sqrt{x}}\lambda_g^2(p)-\sum_{p\leq \sqrt{x}}1 & \geq
\big(1+o(1)\big)\frac{\sqrt{x}}{\log x}\\
& \gg \frac{\sqrt{x}}{\log x},
\end{align*}
provided $x\geq exp \big(\tilde{c}\log ^2(\sqrt{{\rm{max}}\{q(sym^2 f),q(sym^2 g)\}})\big)$.
Therefore, we have
\begin{align}\label{mub}
\sum_{p\leq \sqrt{x}}\lambda_f^2(p)\lambda_g^2(p)\gg \frac{\sqrt{x}}{\log x}.
\end{align}

After substituting estimates \eqref{mlb} and \eqref{mub} in \eqref{mb}, we get the required result and hence completes the proof.
\end{proof}
\vspace{.2in}
\begin{rmk}
In the proof of \propref{lowerb}, if one assume the Ramanujan's bound for the coefficients i.e., $|\lambda_f(p)|,|\lambda_g(p)|\leq2$, then one would get the following:
\begin{equation*}
\sum_{n\leq x}\lambda_f(n)\lambda_g(n)\gg \frac{x}{\log^2x}.
\end{equation*}
\end{rmk}

\bigskip
Now, let $L_{f,g}(s):=\sum_{n=1}^{\infty}\dfrac{\lambda_f^2(n)\lambda_g^2(n)}{n^s}$ for $\rm{Re} (s)\gg 1$. Then we have the following result;
\begin{lem}\label{dec}
For $\rm{Re} (s)>1$, we have
\begin{equation}
L_{f,g}(s)=\zeta(s)L(sym^2f,s)L(sym^g,s)L(sym^2f\otimes sym^2g,s)U(s),
\end{equation}
where $U(s)$ is a Dirichlet series converges uniformly and absolutely in the half plane $\rm{Re} (s)>\frac{15}{16}$.
\end{lem}
\begin{proof}
The proof is exactly same as in the case of holomorphic cusp forms (see, \cite[Lemma 2.2]{lu}) except that the $U$ factor converges absolutely in the region ${\rm{Re}}(s)>\frac{15}{16}$ instead of ${\rm{Re}}(s)>\frac{1}{2}$. Since for a Hecke-Maass cusp form the Kim-Sarnak's bound is weaker than the Ramanujan's bound.
\end{proof}

\begin{prop}\label{cbs}
Let $f$ and $g$ be two distinct normalized Hecke-Maass cusp forms for $SL_2(\mathbb Z)$. Then for any $\epsilon>0$
and $\frac{15}{16}+\epsilon\leq\sigma
< 1+\epsilon$, $t\in \mathbb R$, we have
\begin{equation}
L_{f,g}(\sigma+it)\ll_{f,g,\epsilon}(1+|t|)^{8(1+\epsilon-\sigma)}.
\end{equation} 
\end{prop}

\begin{proof}
The result follows from the convexity bound for the each factor of the $L$-function $L_{f,g}(s)$, which is as follows:
\begin{equation*}
\zeta(\sigma+it)\ll_{\epsilon} (1+|t|)^{\frac{1}{2}(1+\epsilon-\sigma)},
\end{equation*}
\begin{equation*}
L(sym^2f,\sigma+it)\ll_{f,\epsilon} (1+|t|)^{\frac{3}{2}(1+\epsilon-\sigma)},
\end{equation*}
\begin{equation*}
L(sym^2f\otimes sym^2g ,\sigma+it)\ll_{f,g,\epsilon} (1+|t|)^{\frac{9}{2}(1+\epsilon-\sigma)}.
\end{equation*}
Notice that the $U$ factor is absolutely convergent in the given range.
One can obtain the last two inequalities as like that for Riemann zeta function $\zeta(s)$, for that we need to use the functional 
equations for them. For the Gamma factors appearing in the functional equation for the $L$-functions 
$L(sym^2f, s)$ and $L(sym^2f\otimes sym^2g ,s)$ see, \cite[p. 137, eq. 5.100]{ik} and \cite[Theorem 12.1.4, p. 367]{gol}, respectively.
\end{proof}

We are now in a position to give a prove of \thmref{signc}.
\section{Proof of \thmref{signc}}
We give a prove of the first case since the other case can be handle exactly in the same fashion just by replacing $\lambda_f(n)\lambda_g(n)$ with $-\lambda_f(n)\lambda_g(n)$.

A special case of Ramakrishnan's modularity theorem on the Rankin-Selberg $L$-function \cite{rk} states that:
\textit{Suppose that $f$ and $g$ are two distinct Hecke-Maass cusp forms. Then there exists a cuspidal representation $\pi_{f\times g}$ on $GL_4(\mathbb{A}_{\mathbb{Q}})$ such that 
$$L(f\otimes g,s)=L(\pi_{f\times g},s).$$}
So, from \cite[Theorem 1.2]{jm} we have an upper bound 
\begin{equation}\label{1ub}
\sum_{x<n\leq 2x}\lambda_f(n)\lambda_g(n)\ll_{\epsilon} x^{\frac{12}{17}+\epsilon},
\end{equation}
for any $\epsilon>0$.

Now on the contrary, suppose $\lambda_f(n)\lambda_g(n)\geq 0$ for all $n \in (x, 2x]$. Then 
\begin{equation}\label{sl1}
\bigg(\sum_{x<n\leq 2x}\lambda_f(n)\lambda_g(n)\bigg)^2 \gg_{\alpha} \frac{1}{16 x^{\frac{7}{16}+4\alpha}}
\bigg(\sum_{x<n\leq 2x}\lambda_f^2(n)\lambda_g^2(n)\bigg)^2,
\end{equation}
holds for any $\alpha>0.$
On the other hand for $x\geq 2$, 
\begin{equation*}
\sum_{x<n\leq 2x}\lambda_f^2(n)\lambda_g^2(n)\log^2\bigg(\frac{x}{n}\bigg) \leq \log^{2}x\sum_{x<n\leq 2x}\lambda_f^2(n)\lambda_g^2(n).
\end{equation*}
Therefore in order to get a lower bound for the function $\sum_{x<n\leq 2x}\lambda_f(n)\lambda_g(n)$ it is sufficient to obtain a lower bound for $\sum_{x<n\leq 2x}\lambda_f^2(n)\lambda_g^2(n)\log^2\bigg(\dfrac{x}{n}\bigg)$.
From the Perron's formula \cite[p. 228, exercise 169]{ten}, we have
\begin{equation*}
\sum_{n\leq x}\lambda_f^2(n)\lambda_g^2(n)\log^2\bigg(\frac{x}{n}\bigg)=\frac{1}{\pi i}
\int_{1+\epsilon-i\infty}^{1+\epsilon+i\infty}L_{f,g}(s)\frac{x^s}{s^3}ds.
\end{equation*}
Since the integrand function $L_{f,g}(s)\dfrac{x^s}{s^3}$ is analytic in the region ${\rm{Re}}(s)> \frac{63}{64}-\delta$ for some $\delta>0$ except for a simple pole at $s=1$ with non-zero residue $c_{f,g}x$, since expect $\zeta(s)$ all other factors of $L_{f,g}(s)$
make sense and are non-zero at $s=1$ (see, preliminaries).
So if we move the line of integration to the line ${\rm{Re}}(s)=\frac{63}{64}$ and then use \propref{cbs}, we get
\begin{align}
\sum_{n\leq x}\lambda_f^2(n)\lambda_g^2(n)\log^2\bigg(\dfrac{x}{n}\bigg)&=c_{f,g}x+\frac{1}{\pi i}
\int_{\frac{63}{64}-i\infty}^{\frac{63}{64}+i\infty}L_{f,g}(s)\frac{x^s}{s^3}ds \notag\\
&=c_{f,g}x+O\bigg(x^{\frac{63}{64}}\int_{-\infty}^{\infty}\frac{(1+|t|)^{\frac{1}{8}}}{(1+t^2)^{3/2}}dt\bigg) \notag \\
&=c_{f,g}x+O(x^{\frac{63}{64}})\notag.
\end{align}

Note that the integrals over the horizontal segments $\sigma+iT ~(\frac{63}{64}\leq \sigma \leq 1+\epsilon)$ for $T\rightarrow
\infty$ goes to zero, by the convexity bound given in \propref{cbs} to $L_{f,g}(s)$ and due to the presence of the factor $s^3$
in the denominator. Therefore 
\begin{equation}\label{lb2}
\sum_{x<n\leq 2x}\lambda_f^2(n)\lambda_g^2(n)\gg \frac{x}{\log^2 x}.
\end{equation}
Now from \eqref{sl1} and \eqref{lb2}, we get
\begin{equation}\label{lb3}
\sum_{x<n\leq 2x}\lambda_f(n)\lambda_g(n)\gg \frac{x^{\frac{25}{32}-2\alpha}}{\log^2 x}
\end{equation}
which gives a contradiction to our assumption because \eqref{1ub} and \eqref{lb3} are not compatible with each other for suitably chosen $\epsilon >0, \alpha>0$ (specifically, $\epsilon+2\alpha<41/544$), and sufficiently large $x$. This finishes the proof.

\section{proof of \thmref{firstn}}
\noindent
Our idea of the proof is essentially same as that of \cite{ikg}. Throughout the section we assume $f$ and $g$ are as in  \thmref{firstn}.

To prove \thmref{firstn}, first we consider the sum 
\begin{equation*}\label{parf}
S(x):=\sum_{n\leq x}\lambda_f(n)\lambda_g(n)\log^2\bigg(\frac{x}{n}\bigg).
\end{equation*}
Then the desired result will follow from upper and lower bound estimates for $S(x)$
under the assumption that 
\begin{equation}\label{negative}
\lambda_f(n)\lambda_g(n)\geq 0 ~~{\rm{for~ all}} ~~n\leq x.
\end{equation}

\bigskip

From \propref{lowerb}, we have a lower bound estimate for the sum $S(x)$ provided  
$x\geq \exp \big(\tilde{c}\log ^2(\sqrt{{\rm{max}}\{q(sym^2 f),q(sym^2 g)\}})\big)$. The only point remains to achieve an upper bound
for $S(x)$. 
For that we need to get an 
estimate for the function $R(f,g,s)$ near the line ${\rm{Re}}(s)=\frac{1}{2}$. Although the idea is standard but to make the implied constants explicitly we proceed in detail.

\begin{prop}\label{fb}
Let $f$ and $g$ be two distinct normalized Hecke-Maass forms of weight zero for $SL_2(\mathbb{Z})$. Then for any $t\in \mathbb{R}$, we have
\begin{equation}\label{ff}
\zeta\bigg(\frac{5}{2}+2it\bigg)R\bigg(f,g, \frac{5}{4}+it\bigg)\ll 1,
\end{equation}
and 
\begin{equation}\label{mof}
\zeta\bigg(\frac{-1}{2}+2it\bigg)R\bigg(f,g, \frac{-1}{4}+it\bigg)\ll (1+u+v)^3 |1+it|^3.
\end{equation}
\end{prop}

\begin{proof}
The inequality \eqref{ff} trivially holds because the series $\zeta(2s)$ and $R(f,g,s)$ are absolutely convergent in the region ${\rm{Re}}(s)\geq \frac{5}{4}$. To show that the series $R(f,g,s)$ converges absolutely in the region
${\rm{Re}}(s)\geq \frac{5}{4}$, we have used the Kim-Sarnak bound.
 
To derive inequality \eqref{mof}, we use the functional equation \eqref{fe} for the function $L(f\otimes g,s)$.
So from \eqref{fe}, we have
\begin{align*}
\zeta\bigg(\frac{-1}{2}+2it\bigg)R\bigg(f,g, \frac{-1}{4}+it\bigg)
=&(2\pi)^{2-2s}\frac{\Gamma\bigg(\frac{5}{8}+\frac{r}{2}+i\frac{u+v-t}{2}\bigg)}{\Gamma\bigg(\frac{-1}{8}+\frac{r}{2}-i\frac{u+v-t}{2}\bigg)} \frac{\Gamma\bigg(\frac{5}{8}+\frac{r}{2}+i\frac{u-v-t}{2}\bigg)}{\Gamma\bigg(\frac{-1}{8}+\frac{r}{2}-i\frac{u-v-t}{2}\bigg)} \notag\\ 
&\times
\frac{\Gamma\bigg(\frac{5}{8}+\frac{r}{2}-i\frac{u+v+t}{2}\bigg)}{\Gamma\bigg(\frac{-1}{8}+\frac{r}{2}+i\frac{u+v+t}{2}\bigg)}\frac{\Gamma\bigg(\frac{5}{8}+\frac{r}{2}-i\frac{u-v+t}{2}\bigg)}{\Gamma\bigg(\frac{-1}{8}+\frac{r}{2}+i\frac{u-v
+t}{2}\bigg)}\\
&\times
\zeta\bigg(\frac{5}{2}-2it\bigg)R\bigg(f,g, \frac{5}{4}-it\bigg).
\end{align*}

Now using the Stirling's formula \cite[p. 15, problem 8]{leb}, we get
\begin{align}\label{gb}
\frac{\Gamma\bigg(\frac{5}{8}+\frac{r}{2}+i\frac{u+v-t}{2}\bigg)}{\Gamma\bigg(\frac{-1}{8}+\frac{r}{2}-i\frac{u+v-t}{2}\bigg)} \frac{\Gamma\bigg(\frac{5}{8}+\frac{r}{2}+i\frac{u-v-t}{2}\bigg)}{\Gamma\bigg(\frac{-1}{8}+\frac{r}{2}-i\frac{u-v-t}{2}\bigg)}
\frac{\Gamma\bigg(\frac{5}{8}+\frac{r}{2}-i\frac{u+v+t}{2}\bigg)}{\Gamma\bigg(\frac{-1}{8}+\frac{r}{2}+i\frac{u+v+t}{2}\bigg)}\frac{\Gamma\bigg(\frac{5}{8}+\frac{r}{2}-i\frac{u-v+t}{2}\bigg)}{\Gamma\bigg(\frac{-1}{8}+\frac{r}{2}+i\frac{u-r+t}{2}\bigg)}& \notag \\
\ll (1+u+v)^3 |1+it|^3.
\end{align}
Therefore using inequalities \eqref{ff} and \eqref{gb} in the above functional equation, we get inequalities \eqref{mof}.
\end{proof}

Our next proposition gives convexity bound for the Rankin-Selberg $L$-function $L(f\otimes g,s)$ using the Rademacher's \thmref{rthm} and \propref{fb}.

\begin{prop}
Let $f$ and $g$ be two distinct normalized Hecke-Maass cusp forms for $SL_2(\mathbb{Z})$. Then for any $t\in \mathbb{R}$ and $\frac{-1}{4}< \sigma< \frac{5}{4}$, one has
\begin{equation}\label{subc}
\zeta(2\sigma+2it) R(f,g,\sigma+it)\ll (1+u+v)^{2\left(\frac{5}{4}-\sigma\right)} (3+|t|)^{2\left(\frac{5}{4}-\sigma\right)}.
\end{equation}
\end{prop}

\begin{proof}
From \thmref{rst}, the function $\zeta(2s)R(f,g,s)$ 
satisfy all the necessary conditions for \thmref{rthm} with $a=\frac{-1}{4}$ and $b=\frac{5}{4}$.
Furthermore from \propref{fb}, we have 
\begin{equation*}
P=\frac{5}{4},~~~~~~~ E=C_1(1+u+v)^3, ~~~~~~F=C_2
\end{equation*}
\begin{equation*}
\alpha=3,~~~~~~ \beta=0,
\end{equation*}
where $C_1$ and $C_2$ are absolute constants.
After substituting these values in \thmref{rthm}, we get
\begin{equation*}
\zeta(2\sigma+2it) R(f,g,\sigma+it)\ll (1+u+v)^{2\left(\frac{5}{4}-\sigma\right)} (3+|t|)^{2\left(\frac{5}{4}-\sigma\right)},
\end{equation*}
for $\frac{-1}{4}< \sigma< \frac{5}{4}$, and hence this completes the proof.
\end{proof}

\vspace{.2in}
\noindent
After substituting $\sigma=\frac{1}{2}+\delta$ and $\zeta(1+2\delta+2it)^{-1}\ll_{\delta} 1$ for any $\delta>0$ in \eqref{subc}, we get the following immediate corollary.
\begin{cor}\label{mc}
For any $t \in \mathbb{R}$ and any $0<\delta\leq 3/4$, we have
\begin{equation}
R\left(f,g,\frac{1}{2}+\delta+it\right)\ll_{\delta} (1+u+v)^{\left(\frac{3}{2}-2\delta\right)}(3+|t|)^{\left(\frac{3}{2}-2\delta\right)}.
\end{equation}
\end{cor}
\vspace{.5cm}
Now we derive an upper bound for $S(x)$.

\begin{prop}\label{upperb}
Let $f$ and $g$ be two distinct normalised Hecke-Maass forms for $SL_2(\mathbb{Z})$. Then for any $0<\delta\leq 3/4 $, we have
\begin{equation}
S(x)\ll_{\delta} (1+u+v)^{\left(\frac{3}{2}-2\delta\right)}x^{\frac{1}{2}+\delta}.
\end{equation}
\end{prop}

\begin{proof}
From the Perron's formula \cite[p. 228, Exercise 169]{ten}, we have
\begin{equation*}
\sum_{n\leq x}\lambda_f(n)\lambda_g(n)\log^2\bigg(\frac{x}{n}\bigg)=\frac{1}{\pi i}
\int_{5/4-i\infty}^{5/4+i\infty}R(f,g,s)\frac{x^s}{s^3}ds.
\end{equation*}
Since the integrand function $R(f,g,s)\dfrac{x^s}{s^3}$ is analytic in the region ${\rm{Re}}(s)\geq 1/2+\delta$. So if we move the line of integration to the line $1/2+\delta$ and then use \corref{mc}, we get
\begin{align}
S(x)=&\frac{1}{\pi i}
\int_{1/2+\delta-i\infty}^{1/2+\delta+i\infty}R(f,g,s)\frac{x^s}{s^3}ds \notag\\
& \ll(1+u+v)^{\left(\frac{3}{2}-2\delta\right)}x^{\frac{1}{2}+\delta}\int_{-\infty}^{\infty}\frac{(3+|t|)^{\left(\frac{3}{2}-2\delta\right)}}{(1+t^2)^{3/2}}dt \notag \\
& \ll (1+u+v)^{\left(\frac{3}{2}-2\delta\right)}x^{\frac{1}{2}+\delta} \notag.
\end{align}
Hence this completes the proof of the proposition.
\end{proof}
\vspace{.2in}
\noindent
We are now in a position to complete the proof of \thmref{firstn}.
\begin{proof}
Let $x\geq exp \big(\tilde{c}\log ^2(\sqrt{{\rm{max}}\{q(sym^2 f),q(sym^2 g)\}})\big)$ be a real number such that\\ 
$\lambda_f(n)\lambda_g(n)\geq 0$ for all $n\leq x$.
Then from \propref{upperb} and \propref{lowerb}, we infer that
\begin{equation*}
\frac{x^{25/32}}{\log^2x}\ll S(x)\ll (1+u+v)^{\left(\frac{3}{2}-2\delta\right)}x^{\frac{1}{2}+\delta}.
\end{equation*}
\begin{equation*}
x^{9/32-2\delta}\ll (1+u+v)^{\left(\frac{3}{2}-2\delta\right)}.
\end{equation*}
After substituting $\delta=\frac{1}{5208}$ in the last inequality, we get the required result.
\end{proof}

\begin{acknowledgements}
 The authors would like to thank the referee for numerous useful comments and substantial
 corrections.
\end{acknowledgements}

\end{document}